\documentclass[12pt,twoside,reqno]{amsart}

\usepackage{amsmath, amssymb, amsthm, enumitem, esint}
\usepackage[hyperindex]{hyperref}

\usepackage{hyperref}
\hypersetup{bookmarksdepth=3}

\newcommand{\IR}{\mathbb{R}}

\newcommand{\IB}{\mathbb{B}}
\newcommand{\IS}{\mathbb{S}}

\newcommand{\UU}{\mathcal{U}}
\newcommand{\VV}{\mathcal{V}}
\newcommand{\MM}{\mathcal{M}}

\newcommand{\XX}{\mathcal{X}}

\newcommand{\eps}{\varepsilon}
\newcommand{\la}{\lambda}

\newcommand{\ov}[1]{\overline{#1}}
\newcommand{\td}[1]{\widetilde{#1}}

\DeclareMathOperator{\sing}{sing}

\DeclareMathOperator{\spt}{spt}

\newcommand{\EMPTY}[1]{}

\newtheorem{Theorem}[equation]{Theorem}

\newtheorem{Corollary}[equation]{Corollary}

\newtheorem{Claim}[equation]{Claim}

\theoremstyle{definition}

\theoremstyle{remark}

\numberwithin{equation}{section}

\title[Morse resolution of flows with cylindrical singularities]{Morse resolution of mean curvature flows with cylindrical singularities}
\author{Richard H Bamler, Felix Schulze and Lu Wang}
\address{Department of Mathematics, UC Berkeley, CA 94720, USA}
\email{rbamler@berkeley.edu}
\address{Mathematics Institute, University of Warwick, Coventry CV4 7AL, UK}
\email{felix.schulze@warwick.ac.uk}
\address{Department of Mathematics, Yale University, New Haven, CT 06511, USA}
\email{lu.wang@yale.edu}
\date{\today}

\begin{document}

\begin{abstract}
We show that a compact mean curvature flow whose singularities have multiplicity-one cylindrical tangent flows admits a smooth Morse resolution.
The resolution agrees with the spacetime track outside any prescribed neighborhood of its singular set and all its critical points have the expected index.
No nondegeneracy or isolatedness assumption is required.
In the mean-convex case the result follows by smoothing the global arrival function.
\end{abstract}

\maketitle

\section{Introduction}

In recent years, there has been growing interest in understanding the topological change of mean curvature flow in the presence of cylindrical singularities.
This question has also served as one motivation for analytic investigations of nondegenerate and generic cylindrical singularities.
Sun--Wang--Xue~\cite{SunWangXue} describe the evolution of the flow through a single nondegenerate cylindrical singularity, while Sz\'ekelyhidi~\cite{Szekelyhidi} constructs an almost mean curvature flow with only spherical and nondegenerate cylindrical singularities and uses their result to interpret its topological change.
In dimension $2$, the monotonicity of the genus was discussed in~\cite{WhiteGenus}.

The purpose of this note is to clarify that nondegeneracy and isolatedness are not needed for this more basic topological question.
While the results just mentioned remain interesting for their stronger analytic and geometric conclusions, a more elementary construction can be used to understand the topology: the spacetime track can be replaced, in an arbitrarily small neighborhood of its singular set, by a smooth hypersurface on which time is a Morse function of the expected index.

The proof can be summarized as follows:
Near a cylindrical singular point the flow is locally mean-convex and can therefore be described, in a spatial neighborhood in $\IR^{n+1}$, by the level sets of an arrival function $u$.
We smooth $u$ by convolution and then make a small Morse perturbation.
If the required index bound failed at arbitrarily small scales, then a blow-up would give one of the known convex ancient models, whose arrival function has sufficiently many strictly concave directions after convolution.
This gives a contradiction, and the result follows from elementary Morse theory.
In particular, potential accumulation of cylindrical singular points causes no additional difficulty.

Part of this argument relies on the recent full resolution of the Mean Convex Neighborhood Conjecture~\cite{BamlerLai}, which gives the required local mean-convexity and blow-up control for general flows.
This input is not necessary in settings where the behavior near cylindrical singularities is already sufficiently well understood.
For example, in the globally mean-convex case there is a global arrival function and the classical theory of White~\cite{WhiteNature,WhiteSubsequent} and Haslhofer--Kleiner~\cite{HaslhoferKleiner} suffices; for $1$-cylindrical singularities one may instead use the earlier local theory of Choi--Haslhofer--Hershkovits--White~\cite{ChoiHaslhoferHershkovitsWhite}.
In particular, for two-dimensional flows, the genus question in~\cite{WhiteGenus} can also be settled directly from their canonical neighborhood theorem.

\section{Statement of the result}

In the following we use the standard notation and conventions for Brakke flows; the preliminaries of~\cite{BamlerLai} may be used as a reference.

\begin{Theorem}[Morse resolution]\label{Thm_main}
Let $\MM$ be an $n$-dimensional, unit-regular, integral Brakke flow with compact support in $\IR^{n+1}\times[0,T]$.
Suppose that its time-$0$ and time-$T$ slices are regular and given by smooth embedded multiplicity-one hypersurfaces, where the latter is allowed to be empty.
Suppose moreover that at every singular point $q\in\MM^{\sing}$ some tangent flow is a multiplicity-one cylinder $\IR^{k(q)}\times\IS^{n-k(q)}$, where
\[
 0\leq k(q)\leq k_0\leq n-1.
\]
Then given any open neighborhood $\UU$ of $\MM^{\sing}$ in $\IR^{n+1}\times(0,T)$, there is a smooth embedded $(n+1)$-dimensional hypersurface $\XX\subset\IR^{n+1}\times[0,T]$ with the following properties:
\begin{enumerate}[label=\textup{(\alph*)}]
\item $\XX\setminus\UU=(\spt\MM)\setminus\UU$, and $\partial\XX$ consists of the prescribed time-$0$ and time-$T$ slices;
\item the restriction $\tau:=\mathbf{t}|_{\XX}$ of the time function $\mathbf{t}$ is a Morse function, all its critical points lie in $\UU$, and its critical values are distinct;
\item every critical point of $\tau$ has Morse index at least $n-k_0+1$.
\end{enumerate}
The hypersurface $\XX$ may be chosen to lie in an arbitrarily small neighborhood of $\spt\MM$.
\end{Theorem}

Let us describe the conclusion of the theorem in more elementary geometric terms.
Let $0=t_0<t_1<\ldots<t_m<t_{m+1}=T$, where $t_1,\ldots,t_m$ are the critical values of $\tau$, and let $q_i=(p_i,t_i)$ be the corresponding critical points.
For every $i=0,\ldots,m$, the slices
\[
 M_t:=\{x\in\IR^{n+1}:(x,t)\in\XX\},\qquad t\in(t_i,t_{i+1}),
\]
are smooth embedded hypersurfaces moving by isotopy.
Thus their topology does not change between consecutive critical times, and they agree with the original flow outside $\UU$.
Near $q_i$ there are coordinates $y=(y_1,\ldots,y_{n+1})$ on $\XX$ in which
\[
 \tau(y)=t_i-y_1^2-\ldots-y_{j_i}^2
                 +y_{j_i+1}^2+\ldots+y_{n+1}^2,
 \qquad j_i\geq n-k_0+1.
\tag{2.1}\label{eq_Morse_model}
\]
So the two smooth families on either side fit together away from an arbitrarily small neighborhood of $p_i$, and inside that neighborhood they are the level sets of a Morse function with one critical point.
Equivalently, the transition is a handle of index $j_i$, or a $(j_i-1)$-surgery on the $n$-dimensional level hypersurface.
Since the level at $t_i$ itself has a Morse singularity, we may instead cut at regular levels immediately before and after $t_i$ and obtain smooth families on closed intervals whose endpoints are joined by the Morse cobordism~\eqref{eq_Morse_model}.

Notice that the basic model associated with $\IR^k\times\IS^{n-k}$ has Morse index $n-k+1$.
For example, a neck pinch of a surface has index $2$, whereas the disappearance of a spherical component has index $3$.

\medskip

For a compact mean-convex flow, the entire spacetime track is encoded by a global arrival function on the swept-out region.
So in this case one may smooth this function alone and read all topological changes directly from its level sets.
This gives the following simpler formulation.

\begin{Corollary}\label{Cor_mean_convex}
Let $u$ be the arrival function of a compact $n$-dimensional mean-convex flow between two regular levels, the latter possibly empty.
Suppose that every singular point has a cylindrical tangent flow $\IR^k\times\IS^{n-k}$ with $k\leq k_0\leq n-1$.
For every neighborhood $U$ of the singular set in the swept-out region, there is a smooth Morse function $\td u$ that is arbitrarily close to $u$ in the $C^0$-sense and satisfies $\td u=u$ outside $U$.
Every critical point of $\td u$ has index at least $n-k_0+1$.
Its levels describe the topological change by a finite sequence of surgeries.
\end{Corollary}

\medskip

Let us point out a consequence for $k$-convex flows, where $1\leq k\leq n$.
Recall that a hypersurface is $k$-convex if its principal curvatures satisfy $\la_1+\ldots+\la_k\geq0$, with $\la_1\leq\ldots\leq\la_n$.
By~\cite[Theorem~1.17]{CheegerHaslhoferNaber}, the singular tangent flows of the mean curvature flow starting from a smooth compact embedded $k$-convex hypersurface are multiplicity-one cylinders with at most $k-1$ Euclidean directions.
So Theorem~\ref{Thm_main} and Corollary~\ref{Cor_mean_convex} apply with $k_0=k-1$, and all Morse indices are at least $n-k+2$.
Running the level sets of $\td u$ backwards from extinction gives a handle decomposition of the enclosed region with handles of index at most $k-1$.
In particular, in the $2$-convex case this gives the usual description as a union of balls with $1$-handles attached, without constructing a flow with surgery (cf.~\cite{HuiskenSinestrari,BrendleHuisken,HaslhoferKleiner2}).

\medskip

\noindent\textbf{Acknowledgements.}
R.~H. Bamler is supported by NSF grant DMS-2604326.
F.~Schulze has received funding from the European Research Council (ERC) under the European Union's Horizon 1.1 research and innovation programme, grant agreement No.~101200301 (GENREG).
L.~Wang is partially supported by NSF grant DMS-2506407.
We thank the Banff International Research Station (BIRS) for its hospitality.
For the purpose of open access, the authors have applied a Creative Commons Attribution (CC BY) licence to any author accepted manuscript version arising.

\medskip

\noindent\textbf{Statement on the use of AI.}
The mathematical solution presented in this paper arose from a discussion among the authors at the BIRS workshop ``Geometric Flows and Related Topics'' in Banff in August 2026.
AI was used only to assist with writing the paper, not to develop the mathematical ideas.
The manuscript underwent extensive feedback and repeated revisions by the authors.
We have independently read and verified the paper and take full responsibility for its contents.

\section{Proof}

\begin{proof}[Proof of Theorem~\ref{Thm_main}]
Fix a nonnegative radially symmetric $\chi\in C_c^\infty(\IR^{n+1})$ that is positive on $\IB^{n+1}_1$ and satisfies $\int_{\IR^{n+1}}\chi=1$, where $\IB^{n+1}_1$ denotes the unit ball in $\IR^{n+1}$.  For $\eps>0$, set
\[
 \chi_\eps(x):=\eps^{-(n+1)}\chi(x/\eps),\qquad x\in\IR^{n+1}.
\]

By~\cite[Theorem~1.11]{BamlerLai}, every point of $\MM^{\sing}$ has a spacetime neighborhood in which $\spt\MM$ is the graph
\[
 t=u(x)
\tag{3.1}\label{eq_arrival_graph}
\]
of a continuous arrival function\footnote{In fact, the arrival function is continuously differentiable and twice differentiable, but need not be $C^2$; see~\cite{ColdingMinicozziDifferentiability,ColdingMinicozziRegularity}. In an arrival neighborhood the singular set is precisely $\{(x,u(x)):\nabla u(x)=0\}$. Thus the standard Morse perturbation argument and Morse lemma do not apply directly to $u$.} that, at every regular point, is smooth with nonvanishing gradient.
The corresponding arrival functions agree on overlaps.
Endow $\spt\MM$ with the subspace topology induced from $\IR^{n+1}\times[0,T]$.
It follows that there is a neighborhood $\VV$ of $\MM^{\sing}$ in $\spt\MM$ on which the spatial projection
$\pi:\VV\to\IR^{n+1}$, $(x,t)\mapsto x$, is a local homeomorphism.
We equip $\VV$ with the smooth flat structure obtained by pulling back the standard spatial coordinates.
Let $u:=\mathbf{t}|_\VV:\VV\to(0,T)$.
Thus, in each spatial coordinate neighborhood, $u$ is represented by the corresponding arrival function in~\eqref{eq_arrival_graph}.

Choose neighborhoods $\MM^{\sing}\subset\VV'\Subset\VV''\Subset\VV\cap\UU$, open in $\spt\MM$.
For sufficiently small $\eps$, convolution of $u$ with $\chi_\eps$ in the spatial coordinates is defined on $\VV''$ and is well-defined, since all transition maps are the identity.
We use this convolution on $\VV'$, blend it with $u$ via a suitable partition of unity on the regular collar $\VV''\setminus\ov{\VV'}$, and leave $\mathbf{t}$ unchanged outside $\VV''$.
The blending can be arbitrarily small in $C^\infty$ on the collar and hence introduces no critical points.
Denote the resulting function on $\spt\MM$ by $u_\eps$.
The map
\[
 q\longmapsto\big(\pi(q),u_\eps(q)\big),\qquad q\in\spt\MM,
\tag{3.2}\label{eq_smoothed_embedding}
\]
where $\pi$ also denotes the spatial projection on the rest of $\spt\MM$,
is an embedding for small $\eps$ whose image agrees with $\spt\MM$ outside $\UU$ and is arbitrarily close to it.
It remains to control the critical points.

\begin{Claim}\label{Cl_spectral}
For all sufficiently small $\eps>0$, at every critical point $q$ of $u_\eps$ the Hessian $D^2u_\eps(q)$, computed in the flat spatial coordinates, has at least $n-k_0+1$ strictly negative eigenvalues, counted with multiplicity.
\end{Claim}

\begin{proof}
Suppose that the claim was false.
Then we can find a sequence $\eps_i\to0$ and critical points $q_i\in\VV'$ such that $D^2u_{\eps_i}(q_i)$ has at most $n-k_0$ strictly negative eigenvalues.
Set $x_i:=\pi(q_i)$.
After passing to a subsequence, $q_i$ converges to a point of $\MM^{\sing}$ whose tangent flow is a cylinder $\IR^k\times\IS^{n-k}$, where $k\leq k_0$.
Writing $\hat u$ for the representative of $u$ in a spatial coordinate neighborhood containing the $q_i$, set
\[
 \hat u_i(y):=\eps_i^{-2}\big(\hat u(x_i+\eps_i y)-\hat u(x_i)\big).
\tag{3.3}\label{eq_rescale}
\]
Since $\int\chi=1$, in these coordinates we have
\[
 (\chi*\hat u_i)(y)=\eps_i^{-2}
 \big(\hat u_{\eps_i}(x_i+\eps_i y)-\hat u(x_i)\big).
\]
Thus, at the origin,
\[
 \nabla(\chi*\hat u_i)(0)=0,
 \qquad
 D^2(\chi*\hat u_i)(0)=D^2\hat u_{\eps_i}(x_i).
\tag{3.4}\label{eq_rescaled_critical}
\]

By the blow-up theorem~\cite[Theorem~1.12]{BamlerLai}, a subsequence of the rescaled flows converges to either a static multiplicity-one affine plane or a convex ancient flow that is asymptotic to an $(n,k')$-cylinder $\IR^{k'}\times\IS^{n-k'}$, where $k'\leq k$, and belongs to the classification in~\cite[Theorem~1.2]{BamlerLai}.
The planar case is impossible.
Indeed, at the origin we have $0=\nabla(\chi*\hat u_i)(0)=(\chi*\nabla\hat u_i)(0)$.
On the other hand, Brakke's regularity theorem gives smooth convergence of the rescaled flows to the static plane.
Together with local strict mean convexity, this implies that, on the convolution region, the vectors $\nabla\hat u_i$ are almost perpendicular to the limiting plane and point to the same side, so their convolution cannot vanish.
So $\hat u_i$ converges locally uniformly to the arrival function $u_\infty$ of a nonflat ancient model, and
\[
 \chi*\hat u_i\longrightarrow\chi*u_\infty
 \qquad\text{locally smoothly}.
\tag{3.5}\label{eq_smooth_conv}
\]

By~\cite[Proposition~3.2]{BourniLangfordLynch}, the function $u_\infty$ is concave.
Moreover, at every regular point its Hessian has at least $n-k'+1$ strictly negative eigenvalues.
To see this for the cylinder, we can write
\[
 u_\infty(y,z)=c-\frac{|z|^2}{2(n-k')}.
\]
The other models are ancient ovals, bowl solitons and flying wing translating solitons, each times suitable Euclidean factors.
The time slices of these models split off at most $k'-1$ Euclidean factors and are strictly convex in the remaining directions.
On a regular level set, the restriction of $D^2u_\infty$ to its tangent space is $-A/H$, with $A$ computed using the outward normal.
It therefore has at least $n-k'+1$ strictly negative directions.
This also covers the translating models, even though their arrival functions are affine in the translating direction.
Convolution preserves this lower bound on the number of strictly negative eigenvalues.
Indeed, since the distributional Hessian of $u_\infty$ is negative semidefinite and $\chi$ is positive, any vector in the kernel of $D^2(\chi*u_\infty)(0)$ must lie in the kernel of the Hessian of $u_\infty$ on every regular open subset meeting the support of $\chi$.
So this kernel has dimension at most $k'$, which contradicts~\eqref{eq_rescaled_critical} and~\eqref{eq_smooth_conv}.
\end{proof}

Now fix $\eps>0$ small enough such that the claim holds.
Since the condition on the Hessian is open near the compact critical set, a further arbitrarily small perturbation, supported in $\VV''$, makes $u_\eps$ Morse with distinct critical values, without changing the lower bound on the number of negative eigenvalues or creating critical points on the regular collar.
Since this number is a lower bound for the Morse index at a nondegenerate critical point, the resulting Morse indices are at least $n-k_0+1$.
The resulting graph as in~\eqref{eq_smoothed_embedding} gives the desired $\XX$.
The description~\eqref{eq_Morse_model} and the surgery statement follow from the Morse lemma.
\end{proof}

\section{The mean-convex case}

\begin{proof}[Proof of Corollary~\ref{Cor_mean_convex}]
We give the argument separately to explain its simpler form in this setting and to obtain a perturbation of the arrival function itself.
We may convolve the \emph{global} arrival function $u$ directly, so no local patching is needed, apart from blending with $u$ on a regular collar outside the singular set.
If the analogue of Claim~\ref{Cl_spectral} was false, then the same blow-up argument would give a limiting ancient flow; the critical-point condition rules out a planar limit.
The classical theory~\cite{WhiteNature,WhiteSubsequent,HaslhoferKleiner} gives a convex noncollapsed ancient limit whose asymptotic cylinder has at most $k_0$ Euclidean directions, without using~\cite{BamlerLai}.
The limit is either a round shrinking cylinder, or it splits off at most $k_0-1$ Euclidean factors and is strictly convex in the other directions.
Indeed, splitting off $k_0$ factors would leave a flow with a spherical backward limit, which is itself a family of round shrinking spheres by Huisken's monotonicity formula.
Its arrival function is concave by~\cite[Proposition~3.2]{BourniLangfordLynch}, and the same computation on regular level sets, together with the explicit cylindrical formula, gives at least $n-k_0+1$ negative Hessian directions.
Positivity of the convolution kernel preserves this bound.
A small Morse perturbation, supported in $U$, proves the corollary.
\end{proof}

\section{The two-dimensional case}

For $n=2$, Theorem~\ref{Thm_main} gives Morse indices at least $2$: the only changes are neck pinches and disappearances of spherical components, so genus cannot increase between regular times.
The monotonicity of the genus of the regular part, including at singular times, can also be seen directly via the canonical neighborhood theorem~\cite[Corollary~1.18]{ChoiHaslhoferHershkovitsWhite}.

For $\eps>0$ to be determined below, fix a corresponding canonical neighborhood scale $r>0$.
It suffices to compare arbitrary times $t_1<t_2$ with $t_2-t_1$ sufficiently small compared with $r^2$.
We choose a closed neighborhood of the singular set in $\spt\MM\cap(\IR^3\times[t_1,t_2])$.
We denote the time-slices of this neighborhood by $M'_t$ and require that all regular points of $M'_t$ are covered by canonical neighborhoods.
In these neighborhoods, the flow is $\eps$-close, after rescaling to unit mean curvature, to a round shrinking sphere or cylinder, a bowl soliton or an ancient oval.
We can moreover arrange that the boundaries $\partial M'_t$ consist of circles that are cross-sections of $\eps$-necks of scale comparable to $r$ and are transported by the normal flow.
By the choice of $t_2-t_1$, these necks remain regular throughout the interval, and the complements $M''_t:=\ov{(\spt\MM)_t\setminus M'_t}$ vary by isotopy.
For every $t\in[t_1,t_2]$, the components of the regular part of $M'_t$ are spheres, neck-covered tori, or disks and annuli covered by $\eps$-necks and $\eps$-caps, with some boundary circles possibly replaced by horn ends approaching the singular set.

Let $C$ be an annular component of the regular part of $M'_{t_2}$ whose two boundary circles lie in $\partial M'_{t_2}$ and hence are adjacent to $M''_{t_2}$.
Thus $C$ is compact, has no horn ends, and is covered by $\eps$-necks.
Follow $C$ backwards by the normal flow.
Its boundary circles follow $\partial M'_t$, and the trajectories cannot cross this moving boundary, so they remain in the canonical neighborhoods of $M'_t$.
In all models, neck points remain in necks backwards in time and their curvature decreases, so for $\eps$ sufficiently small this also holds in our canonical neighborhoods.
The backward curvature bound and compactness exclude singular endpoints.
Consequently, $C$ remains a compact annular component with both boundary circles contained in $M''_t$ throughout $[t_1,t_2]$; in particular, no horn end or cap can develop.
Distinct annuli remain disjoint by uniqueness of the normal flow.
Neck-covered tori likewise persist backwards.

$M''_{t_1}$ and $M''_{t_2}$ can be identified by isotopy.
The other annular components have at least one horn end and are attached to $M''_t$ along at most one boundary circle, so deleting them does not change the total genus; the same holds for disk and spherical components.
The only changes from $t_1$ to $t_2$ that can affect genus are replacements of annuli attached at both ends by disks or horn-ended annuli, and disappearances of neck-covered tori.
None of these changes increases the total genus.

\end{document}